\documentclass[12pt]{amsart}
\usepackage{txfonts}
\usepackage{subfig}
\usepackage{amscd,amssymb,mathabx}
\usepackage[arrow,matrix,graph,frame,poly,arc,tips]{xy}
\usepackage{graphicx,calrsfs}
\usepackage{mathtools}
\usepackage{tikz}
\usetikzlibrary{decorations.markings}
\usetikzlibrary{shapes}
\usetikzlibrary{backgrounds}
\usetikzlibrary{calc}
\usepackage{mdwlist}
\usepackage{multirow}
\usepackage{enumerate}
\usepackage{verbatim}

\newcommand\ba{\begin{align*}}
\newcommand\ea{\end{align*}}
\newcommand\be{\begin{enumerate}}
\newcommand\ee{\end{enumerate}}
\newcommand\bp{\begin{proof}}
\newcommand\ep{\end{proof}}
\newcommand\bpp{\begin{prop}}
\newcommand\epp{\end{prop}}
\newcommand\bpb{\begin{prob}}
\newcommand\epb{\end{prob}}
\newcommand\bd{\begin{defn}}
\newcommand\ed{\end{defn}}
\newcommand\bh{\begin{hint}}
\newcommand\eh{\end{hint}}

\newcommand\Q{\mathbb{Q}}

\newcommand\R{\mathbb{R}}

\newcommand\Z{\mathbb{Z}}
\newcommand\bH{\mathbb{H}}

\newcommand\CC{\mathcal{C}}

\newcommand\Sym{\operatorname{Sym}}

\newcommand\supp{\operatorname{supp}}

\newcommand\gam{\Gamma}
\newcommand\Mod{\operatorname{Mod}}
\newcommand\PSL{\operatorname{PSL}}
\newcommand\SL{\operatorname{SL}}

\DeclareMathOperator\Homeo{Homeo}

\newcommand\yt{\widetilde}

\DeclareMathOperator\Fix{Fix}
\DeclareMathOperator\Out{Out}
\DeclareMathOperator\Aut{Aut}

\DeclareMathOperator\Diff{Diff}

\def\thetitle{{Actions of right-angled Artin groups in low dimensions}}
\def\theauthors{{Thomas Koberda}}
\usepackage{hyperref}
\hypersetup{
   colorlinks=false,
   plainpages,
   urlcolor=black,
   linkcolor=black
   pdftitle=   \thetitle,
   pdfauthor=  {\theauthors}
}

\theoremstyle{theorem}
\newtheorem{thm}{Theorem}[section]
\newtheorem{lem}[thm]{Lemma}
\newtheorem{cor}[thm]{Corollary}
\newtheorem{prop}[thm]{Proposition}

\newtheorem{que}[thm]{Question}
\newtheorem*{claim*}{Claim}

\theoremstyle{remark}

\newtheorem{rem}[thm]{Remark}

\theoremstyle{definition}
\newtheorem{defn}[thm]{Definition}
\newtheorem{prob}{Problem}[section]

\begin{document}
\title\thetitle
\date{\today}
\keywords{}

\author[T. Koberda]{Thomas Koberda}
\address{Department of Mathematics, University of Virginia, Charlottesville, VA 22904-4137, USA}
\email{thomas.koberda@gmail.com}
\urladdr{http://faculty.virginia.edu/Koberda}

\begin{abstract}
We survey the role of right-angled Artin groups in the theory of diffeomorphism groups of low dimensional manifolds. We first describe some of the subgroup structure of right-angled Artin groups. We then discuss the interplay between algebraic structure, compactness, and regularity for group actions on one--dimensional manifolds. For compact one--manifolds, every right-angled Artin group acts faithfully by $C^1$ diffeomorphisms, but the right-angled Artin groups which act faithfully by $C^2$ diffeomorphisms are very restricted. For the real line, every right-angled Artin group acts faithfully by $C^{\infty}$ diffeomorphisms, though analytic actions are again more limited. In dimensions two and higher, every right-angled Artin group acts faithfully on every manifold by $C^{\infty}$ diffeomorphisms. We give applications of this discussion to mapping class groups of surfaces and related groups.
\end{abstract}

\maketitle



\tableofcontents

\section{Introduction}\label{sec:intro}

Let $\Gamma$ be a finite simplicial graph, with vertex set $V=V(\Gamma)$ and edge set $E=E(\gam)$. The right-angled Artin group $A(\gam)$ on $\gam$ is defined to be the group \[A(\gam)=\langle V\mid [v,w]=1 \textrm{ if and only if } \{v,w\}\in E\rangle.\]

In recent years, right-angled Artin groups have been of central importance in geometric group theory (see~\cite{Charney2007,KoberdaNotes}, for some surveys). Among the general reasons for this is the fact that right-angled Artin groups admit so many incarnations. For one, they are a prototypical example of a CAT(0) group, and thus play a central role in the geometry of non-positively curved groups. The work of many authors over many years has illustrated deep connections between CAT(0) geometry and hyperbolic geometry. This culminated in Agol's resolution of the virtual Haken and virtual fibering conjecture~\cite{Agol2008,Agol2013}, in which right-angled Artin groups played a key role. Moreover, right-angled Artin groups exhibit remarkable ubiquity, arising as subgroups of many other groups of classical interest to geometric group theorists such as Coxeter groups, Artin groups, mapping class groups, and diffeomorphism groups of manifolds.

This ubiquity of right-angled Artin groups has aided the study of other classes of groups, because right-angled Artin groups are simultaneously complicated enough to admit a rich structure and also simple enough to be amenable to study.

This article serves as a survey of the role of right-angled Artin groups in the study of diffeomorphism groups of manifolds, concentrating on dimension one. While the author strives to be a good scholar, the results discussed herein are undoubtedly colored by the author's personal tastes, which center around the interplay of algebraic structure, dimension, regularity, and compactness. Only a limited number of proofs or sketches of proofs has been given, though the author strives to include detailed references. When given, proofs follow the original source material, unless otherwise noted.

\section{Terminology}

In this section we collect some basic terminology concerning right-angled Artin groups and diffeomorphism groups which we will require. Most of this terminology is standard, and we include this section for reference. We will identify vertices of $\gam$ with generators of $A(\gam)$. An element $w\in A(\gam)$ is called \emph{reduced} if it is a shortest element in the free group on $V$ which is equal to $w$ in $A(\gam)$. By a result of Hermiller--Meier~\cite{hm1995}, an element of $A(\gam)$ is reduced if and only if one cannot shorten the length of $w$ by switching the order of vertices which are adjacent in $\gam$ and by applying free reductions.

By a \emph{subgraph} $\Lambda$ of a graph $\gam$, we always mean a \emph{full subgraph}, which is to say if $v,w\in V(\Lambda)$, then $\{v,w\}\in E(\Lambda)$ if and only if $\{v,w\}\in E(\gam)$. We will often write $\Lambda\subset\gam$ in this case.

Let $\gam$ be a simplicial graph. The \emph{complement} $\gam^c$ of $\gam$ is the graph whose vertices are the vertices of $\gam$ and whose edges are exactly the non--edges of $\gam$. A \emph{join} $J$ is a simplicial graph for which there exist nonempty subgraphs $J_1,J_2\subset J$ such that \[J=\left(J_1^c\coprod J_2^c\right)^c.\]

If $\gam$ is a (not necessarily finite) simplicial graph, we write $\gam_k$ for the \emph{clique graph} of $\gam$. The vertices of $\gam_k$ are nonempty complete subgraphs of $\gam$, and two vertices $K_1$ and $K_2$ of $\gam_k$ are adjacent if $K_1\cup K_2$ also span a complete subgraph of $\gam$. Note that $\gam$ is naturally a subgraph of $\gam_k$.

If $w\in A(\gam)$ is freely reduced, the \emph{support} of $w$ is denoted $\supp w$ and is defined to be the smallest subgraph $\Lambda$ of $\gam$ such that $w\in A(\Lambda)$.

A freely reduced element $w\in A(\gam)$ is \emph{cyclically reduced} if $w$ is shortest in its conjugacy class in $A(\gam)$.

Let $X$ and $Y$ be metric spaces. A function $f\colon X\to Y$ is called a ($K,C$)--\emph{quasi--isometric embedding} if there exist constants $K\geq 1$ and $C\geq 0$ such that for all $a,b\in X$, we have \[\frac{1}{K}\cdot d_X(a,b)-C\leq d_Y(f(a),f(b))\leq K\cdot d_X(a,b)+C.\] The map $f$ is called a \emph{quasi--isometry} if moreover it is \emph{quasi--surjective}, which is to say that there is a constant $D\geq 0$ such that for all $y\in Y$ there exists an $x\in X$ such that $d_Y(f(x),y)\leq D$. A \emph{quasi--geodesic} is the image of a geodesic under a quasi--isometric embedding.

If $M$ is an orientable manifold, we will write $\Homeo^+(M)$ for the group of orientation--preserving homeomorphisms of $M$. For $1\leq k\leq\infty$, we write $\Diff^k(M)$ for the group of orientation--preserving $C^k$ diffeomorphisms of $M$.

Let $X$ be a set. We write $\Sym(X)$ for the set of bijections $X\to X$. The \emph{support} of an element $\phi\in\Sym(X)$ is defined to be the set of points $x\in X$ such that $\phi(x)\neq x$, and is denoted $\supp\phi$. The \emph{fixed set} of $\phi$ is written $\Fix\phi$ and is defined to be $X\setminus\supp\phi$.

If $S$ is an orientable surface, we write $\Mod(S)$ for the mapping class group of $S$, i.e. the group of orientation--preserving homeomorphisms of $S$ considered up to homotopy.

\section{Some subgroups of right-angled Artin groups}

From an algebraic point of view, right-angled Artin groups are interesting because of their rich array of subgroups. In this section, we give a brief summary of some of the main classes of subgroups of right-angled Artin groups (see~\cite{KoberdaNotes} for a more detailed discussion).

\subsection{Solvable groups}\label{subsec:solvable}

The structure of solvable subgroups of right-angled Artin groups is tightly restricted by CAT(0) geometry (see Subsection~\ref{subsec:vspecial} below). Specifically, the Flat Torus Theorem (see~\cite{BH1999}) implies that every solvable subgroup of a right-angled Artin group is abelian. Precisely, we have the following:

\begin{prop}
Let $G<A(\gam)$ be a solvable subgroup. Then $G\cong\Z^n$, where $n$ is bounded above by the size of the largest complete subgraph of $\gam$. This bound is sharp.
\end{prop}

\subsection{Free groups and convex cocompactness}

Free subgroups of right-angled Artin groups are extremely common. Much like the structure of solvable subgroups, the ubiquity of free subgroups of right-angled Artin groups is dictated by CAT(0) geometry. CAT(0) groups enjoy a \emph{strong Tits alternative}, which is illustrated by the following result of Sageev--Wise:

\begin{thm}[\cite{SageevWise05}]\label{thm:sagwise05}
Let $G$ be a group acting properly on a finite dimensional CAT(0) cube complex. Suppose that there is a uniform bound on the size of finite subgroups of $G$. Then either $G$ contains a nonabelian free group, or $G$ is finitely generated and virtually abelian.
\end{thm}

Again, the reader is directed to Subsection~\ref{subsec:vspecial} below for relevant definitions in the statement of Theorem~\ref{thm:sagwise05}.

For right-angled Artin groups, one can be even more precise. The following is a classical result of Baudisch~\cite{Baudisch1981}, which has since been reproved by a number of authors (see~\cite{CarrThesis,KK2015}, cf.~\cite{AFW2015}):

\begin{thm}
Let $G<A(\gam)$ be a two--generated subgroup. Then $G$ is either abelian or free.
\end{thm}

There are many different ways in which free groups can sit inside of right-angled Artin groups. Of particular interest are free subgroups which play the role of \emph{convex cocompact subgroups}, by analogy to mapping class groups (see~\cite{MR1914566,KL2008,Ham2005}). We will not give a detailed discussion of these subgroups nor the analogy with mapping class groups here, since it would take us rather far afield. We will just state one result about them. An element $1\neq w\in A(\gam)$ is called \emph{loxodromic} if after cyclic reduction, $\supp w$ is not contained in a subjoin of $\gam$ (see~\cite{BC2012}, cf.~\cite{KK2013b}). If $|V(\gam)|\geq 2$ and $\gam$ is connected (which we will assume for the rest of this subsection), loxodromic elements are exactly the elements with cyclic centralizer in $A(\gam)$. A subgroup $G<A(\gam)$ is \emph{purely loxodromic} if every non--identity element of $G$ is loxodromic.

In the following result due to the author jointly with J. Mangahas and S. Taylor~\cite{KMT2014}, the graph $\gam^e$ is the \emph{extension graph} of $\gam$ (see Subsection~\ref{subsec:raag} below).

\begin{thm}\label{thm:ploxo}
Let $G<A(\gam)$ be a finitely generated subgroup. The following are equivalent:
\begin{enumerate}
\item
The group $G$ is purely loxodromic;
\item
The orbit map $G\to\gam^e$ is a quasi--isometric embedding;
\item
The group $G$ is stable.
\end{enumerate}

Moreover, the group $G$ is free.
\end{thm}

Here, a finitely generated subgroup $H$ of a finitely generated group $G$ is \emph{stable} if $K$--quasi--geodesics in $G$ between points of $H$ are uniformly bounded distance from each other. M. Durham and S. Taylor~\cite{DurhamTaylor2015} developed concept and proved that stable subgroups of mapping class groups are exactly convex cocompact subgroups.
In light of Theorem~\ref{thm:ploxo}, finitely generated purely loxodromic subgroups of right-angled Artin groups play the exact analogue of convex cocompact subgroups of right-angled Artin groups.

\subsection{Other right-angled Artin groups}\label{subsec:raag}

It is not very difficult to see that if $\Lambda\subset\gam$ is a subgraph then $A(\Lambda)<A(\gam)$. However, it is clear even in the case where $A(\gam)$ is a nonabelian free group that generally there are more right-angled Artin subgroups of $A(\gam)$ than just the ones coming from subgraphs of $\gam$. The question of which right-angled Artin groups occur as subgroups of $A(\gam)$ was first systematically studied by Kim and the author~\cite{KK2013,KK2013b,KK2015}, cf.~\cite{CDK2013,LeeLee2017}. We outline some of the main features of this theory.

The central object in this theory is the \emph{extension graph} $\gam^e$ of $\gam$, which is the analogue of the curve graph for the right-angled Artin group (see Section \ref{sec:mcg} below). The graph $\gam^e$ is defined by $V(\gam^e)=\{v^g\mid v\in V(\gam),\, g\in A(\gam)\}$, where here $v^g=g^{-1}vg$. We write $\{v^g,w^h\}\in E(\gam^e)$ if $[v^g,w^h]=1$ in $A(\gam)$. The extension graph is generally a locally infinite quasi--tree of infinite diameter, and is quasi--isometric to the \emph{contact graph} for a right-angled Artin group as developed by M. Hagen~\cite{Hagen2014} (cf.~\cite{KK2013b}).

\begin{thm}\label{thm:gex}
Let $\gam$ be a finite simplicial graph.
\begin{enumerate}
\item
If $\Lambda\subset\gam^e$ is a finite subgraph then $A(\Lambda)<A(\gam)$;
\item
If $A(\Lambda)<A(\gam)$ then $\Lambda\subset (\gam^e)_k$;
\item
If $\gam$ has no triangles and $A(\Lambda)<A(\gam)$, then $\Lambda\subset\gam^e$.
\end{enumerate}
\end{thm}

Casals-Ruiz--Duncan--Kazachkov~\cite{CDK2013} have shown that in general there may be right-angled Artin subgroups of $A(\gam)$ whose defining graphs do not occur as subgraphs of $\gam^e$.

\subsection{Surface groups}

It was first observed by Servatius--Droms--Servatius~\cite{SDS1989} that if $C_n$ denotes the graph which is a cycle on $n\geq 5$ vertices then $A(C_n)$ contains the fundamental group of a closed surface of negative Euler characteristic. This result was generalized by Crisp--Wiest~\cite{CW2004}, who developed the method of label--reading maps and proved that with all but the single exception of a non--orientable surface of Euler characteristic $-1$, every hyperbolic surface group occurs as a subgroup of some right-angled Artin group. The dual question of when a particular right-angled Artin group contains a closed hyperbolic surface group is still open in general, though significant progress has been made (see~\cite{CSS2008}, for instance).

\subsection{Coherence}

Despite their innocuous--looking presentations and the relatively tame subgroups discussed up to now, right-angled Artin groups can have rather wild subgroups. Recall that a group $G$ is \emph{coherent} if every finitely generated subgroup of $G$ is finitely presented. Classical examples of coherent groups are free groups, surface groups, and $3$--manifold groups~\cite{ScottCore}. A classical example of an incoherent group if a product $F_2\times F_2$ of two nonabelian free groups.

Right-angled Artin groups are generally incoherent, as is witnessed by the following result of Droms:

\begin{thm}[\cite{Droms87}]
Let $\gam$ be a finite simplicial graph. The right-angled Artin group $A(\gam)$ is coherent if and only if every circuit in $\gam$ of length more than three has a chord.
\end{thm}

In other words, a right-angled Artin group is coherent if and only if the defining graph has no full subgraph which is a cycle of length at least four. Relating the coherence discussion to $3$--manifold groups, Droms also proves that a right-angled Artin group is a $3$--manifold group if and only if each component of the defining graph is a triangle or a tree~\cite{Droms87}.

Subgroups of right-angled Artin groups enjoying (or lacking) various subtler finiteness properties were studied by Bestvina--Brady~\cite{BestBra97}. We will not discuss these subgroups here, since it would take us too far afield: suffice it to note that the subgroups structure of right-angled Artin groups can be very complicated.

Observe that, in light of the discussion in this subsection, the existence of right-angled Artin subgroups of some other group can be useful in proving that the larger group is incoherent.

\subsection{Hyperbolic manifolds and virtual specialness}\label{subsec:vspecial}

The concept of a \emph{virtually special} cube complex was developed by Haglund--Wise~\cite{HW2008}, and provides an intrinsic geometric characterization of subgroups of right-angled Artin groups.

A geodesic metric space $X$ is \emph{CAT(0)} if geodesic triangles in $X$ satisfy the CAT(0) inequality. The reader is directed to~\cite{GromovBook} and~\cite{BH1999} for a complete introduction to CAT(0) geometry. Specifically, let $\{p,q,r\}$ be vertices of a geodesic triangle $T$, and let $\{P,Q,R\}$ be vertices of a geodesic triangle in $\R^2$ (i.e. the comparison triangle) whose side lengths are the same as those of $T$. If $x$ and $y$ are points on the sides $\overline{pq}$ and $\overline{qr}$ respectively then the distance between $x$ and $y$ is at most that between the corresponding points $X$ and $Y$ on $\overline{PQ}$ and $\overline{QR}$ respectively.

A \emph{$k$--cube} is defined to by $[-1,1]^k$, with the standard metric inherited from $\R^k$. A \emph{subcube} is obtained by restricting some coordinate to $\{\pm 1\}$, and \emph{midcube} is obtained by restricting one coordinate to $0$. A \emph{dual cube} to a midcube is a $1$--subcube given by restricting all the nonzero coordinates of the midcube to $\{\pm 1\}$.

A \emph{cube complex} is a union of cubes, glued to each other via isometries of subcubes. A \emph{hyperplane} in a cube complex is a connected subspace which intersects each cube in either a midcube or emptily. A cube complex $X$ is \emph{locally CAT(0)} if the link of every $0$--cube is a flag complex. By a result of Gromov~\cite{GromovBook}, a simply connected locally CAT(0) cube complex is CAT(0).

A locally CAT(0) cube complex $X$ is called \emph{special} (see~\cite{HW2008,Wise2011,WiseERAMS,Wise2012}) if the hyperplanes of $X$ avoid certain pathological configurations. Precisely, we require:

\begin{enumerate}
\item
There are no one--sided hyperplanes;
\item
There are no self--intersecting hyperplanes;
\item
There are no self--osculating hyperplanes;
\item
There is no pair of osculating hyperplanes.
\end{enumerate}

Here, we say that hyperplanes $H_1$ and $H_2$ \emph{osculate} if there are dual subcubes to $H_1$ and $H_2$ which intersect nontrivially.

We say $X$ is \emph{virtually special} if some finite cover of $X$ is special. If $G=\pi_1(X)$ for $X$ (virtually) special, we say that $G$ is (virtually) special.

A prototypical example of a special cube complex is the \emph{Salvetti complex} of a finite simplicial graph $\gam$, denoted $S(\gam)$. To build $S(\gam)$, we start with the one--skeleton of the unit cube in $\R^{|V(\gam)|}$, and we label the unit vectors emanating from the origin by the elements of $V(\gam)$. If $\Lambda\subset\gam$ is complete, we include the face of the cube spanned by the vectors corresponding to the vertices of $\Lambda$. Doing this for each complete subgraph of $\gam$, we obtain a subcomplex $Y$ of the unit cube, equipped with a natural cube complex structure. We then define $S(\gam)$ to be the image of $Y$ in $\R^{|V(\gam)|}/\Z^{|V(\gam)|}$. It is easy to check that $\pi_1(S(\gam))\cong A(\gam)$, and that $S(\gam)$ is a locally CAT(0) cube complex (cf.~\cite{BradyRileyShort}).

A map between cube complexes is a \emph{local isometry} if the induced map on links of $0$--complexes is an inclusion, adjacency preserving, and full (as a map of simplicial complexes). Note that a local isometry of cube complexes induces an inclusion on fundamental groups.

The following is the fundamental result about special cube complexes:

\begin{thm}[\cite{HW2008}]\label{thm:wise}
A locally CAT(0) cube complex is special if and only if it admits a local isometry to $S(\gam)$ for some finite simplicial graph $\gam$.
\end{thm}

Theorem \ref{thm:wise} gives an intrinsic geometric characterization of subgroups of right-angled Artin groups.

Combining the work of Wise and Agol~\cite{Agol2013}, we obtain the following:

\begin{thm}\label{thm:vhaken}
Let $G$ be the fundamental group of a finite volume hyperbolic $3$--manifold. Then $G$ is virtually special.
\end{thm}

There is a generalization of Theorem \ref{thm:vhaken} in higher dimensions due to Bergeron--Haglund--Wise~\cite{BHW2011}, which shows that many lattices in higher dimensional groups of isometries of real hyperbolic space are also virtually special. Thus, right-angled Artin groups generally contain a profusion of (relatively) hyperbolic groups.

\subsection{Computational complexity}
We remark some work of M. Bridson~\cite{MR3151642}, which makes precise the mantra that even finitely presented subgroups of right-angled Artin groups can be very complicated.

\begin{thm}\label{thm:bridson}
There exists a finite simplicial graph $\gam$ and a finitely presented subgroup $H<A(\gam)$ such that:
\begin{enumerate}
\item
The isomorphism problem is unsolvable for the class of finitely presented subgroups of $A(\gam)$;
\item
The conjugacy problem for $H$ is unsolvable.
\end{enumerate}
\end{thm}

Theorem \ref{thm:bridson} shows that any class of groups which is rich enough to contain sufficiently complicated right-angled Artin groups is computationally intractable. Note that, by virtue of residual finiteness (cf. Subsection~\ref{sssec:criterion}), the word problem for right-angled Artin groups and all of their subgroups is automatically solvable (see~\cite{LS2001}, for instance).

\subsection{Some non--subgroups}\label{subsec:nonex}

Other than nonabelian solvable groups as in Subsection~\ref{subsec:solvable}, there are several other classes of groups which cannot occur as subgroups of right-angled Artin groups.

\subsubsection{Finite subgroups}

It is not difficult to prove that a simply connected CAT(0) space is contractible. It follows that right-angled Artin groups are torsion--free, since the universal cover of the Salvetti complex admits the structure of a finite dimensional contractible CW complex.

\subsubsection{Central extensions}

A \emph{central extension} of a group $Q$ is an exact sequence \[1\to Z\to G\to Q\to 1,\] where the image of $Z$ under the second map lies in the center of $G$. The extension is \emph{split} of there exists a map $Q\to G$ splitting this sequence. The extension is \emph{virtually split} if some finite index subgroup of $G$ has the structure of a split extension.

\begin{prop}[See~\cite{Koberda2012,KS2016,KapLeeb1996}]\label{prop:vsplit}
Let $G$ be a non--virtually split central extension. Then no finite index subgroup of $G$ occurs as a subgroup of a right-angled Artin group.
\end{prop}

The Birman Exact Sequence~\cite{MR0375281,FM2012} shows that if $S$ is a surface of genus at least three or genus two with at least two punctures or boundary components, then the mapping class group $\Mod(S)$ contains a copy of the group $U$, the fundamental group of the unit tangent bundle of a closed genus two surface. This group has the presentation \[U=\langle a,b,c,d,z\mid [a,b][c,d]=z^2,\, [a,z]=[b,z]=[c,z]=[d,z]=1\rangle,\] and is a non--virtually split central extension of a closed surface group. By Proposition~\ref{prop:vsplit}, the groups $U$ and $\Mod(S)$ and their finite index subgroups cannot embed in any right-angled Artin group.

\subsubsection{A criterion for admitting a nontrivial homomorphism to a right-angled Artin group}\label{sssec:criterion}

Another useful property of right-angled Artin groups is that they are \emph{residually torsion--free nilpotent}. This means that for every nontrivial element $w\in A(\gam)$, there exists a torsion--free nilpotent quotient of $A(\gam)$ in which $w$ survives (see~\cite{MKSBook}, for instance). In particular, this gives another proof that right-angled Artin groups are torsion--free.

\begin{prop}\label{prop:homomorphism}
Let $G$ be a finitely generated group and let $A(\gam)$ be a right-angled Artin group.
\begin{enumerate}
\item
The group $G$ admits a nontrivial homomorphism to $A(\gam)$ if and only if $H^1(G,\Q)\neq 0$.
\item
If there exists an injective homomorphism $G\to A(\gam)$ and $G$ is nonabelian then $H^1(G,\Q)$ has rank at least two.
\end{enumerate}
\end{prop}
\begin{proof}
For the first item, let $\phi\colon G\to A(\gam)$ be a nontrivial homomorphism. Then there is a torsion--free nilpotent quotient $q\colon A(\gam)\to N$ such that $q\circ\phi$ is not trivial. Since the image of $q\circ\phi$ is torsion--free and nilpotent, the image surjects to $\Z$, whence $H^1(G,\Q)\neq 0$. Conversely, suppose that $H^1(G,\Q)\neq 0$, and choose a surjective homomorphism $\phi\colon G\to\Z$. Choosing any nontrivial element $w\in A(\gam)$, we get a nontrivial map $G\to A(\gam)$ by composing $\phi$ with the map $\Z\to A(\gam)$ sending a generator of $\Z$ to $w$.

For the second item, we again appeal to the residual torsion--free nilpotence of $A(\gam)$ to find a nonabelian torsion--free nilpotent quotient $N$ of $G$. By passing to a further quotient of $G$ if necessary, we may assume that $N$ is a non--split central extension \[1\to Z\to N\to A\to 1,\] where $A$ is torsion--free and abelian. Since equivalence classes of such central extensions are classified by cohomology classes in $H^2(A,Z)$ (see~\cite{AlperinBell,MacLaneHomology}, for instance), this extension is split unless $H^2(A,Z)\neq 0$ . Since $A$ is torsion--free, it follows that it must have rank at least two, whence $H^1(G,\Q)$ has rank at least two.
\end{proof}

From Proposition~\ref{prop:homomorphism}, one can deduce a number of corollaries. For instance, one can show that a noncyclic knot group cannot embed into a right-angled Artin group. Item (2) of Proposition~\ref{prop:homomorphism} is hardly an if and only if criterion, and hence of somewhat limited applicability.

\section{Homeomorphism groups and orderability}

Homeomorphism groups of one--manifolds have been studied classically, and there are satisfying algebraic characterizations of countable subgroups of these homeomorphism groups. General references for this section are~\cite{DDRW08,DeroinNavasRivas,Ghys2001,Navas2011}.

First, note that $\R$ and $S^1$ are both topological groups in their own right, and hence act on themselves by the regular action. We thus obtain many abelian subgroups of $\Homeo^+(\R)$ and $\Homeo^+(S^1)$. Characterizations of $\R$ and $S^1$ as subgroups of $\Homeo^+(\R)$ and $\Homeo^+(S^1)$ respectively are known as H\"older's Theorem and Denjoy's Theorem, respectively, which we record for use in the sequel. Recall that a group $G$ acts freely on a set $X$ if for each $1\neq g\in G$ we have $\Fix g=\varnothing$.

\begin{thm}[H\"older's Theorem]\label{thm:holder}
Let $G<\Homeo^+(\R)$ be a countable group which acts freely on $\R$. Then $G$ is abelian.
\end{thm}

H\"older's Theorem can be improved to say that $G$ is conjugate into the additive group $\R$, provided the action is assumed to by by $C^2$ diffeomorphisms.

\begin{thm}[Denjoy's Theorem]\label{thm:denjoy}
Let $G<\Homeo^+(S^1)$ be a countable group which acts freely on $S^1$. Then $G$ is abelian.
\end{thm}

As with H\"older's Theorem, if the action is assumed to be by $C^2$ diffeomorphisms then $G$ is conjugate into the group $S^1$.

We turn our attention general group actions on the line and on the circle. A group $G$ is \emph{left orderable} if there exists a total ordering $\leq$ on $G$ which is left invariant, i.e. for all $a,b,g\in G$ we have $a\leq b$ if and only if $ga\leq gb$. A group $G$ is \emph{cyclically orderable} if there exists a function $\omega\colon G^3\to \{\pm 1\cup 0\}$ which is:
\begin{enumerate}
\item
Left invariant;
\item
\emph{Cyclic}, i.e. $\omega(a,b,c)=\omega(b,c,a)$;
\item
\emph{Asymmetric}, i.e. $\omega(a,b,c)=-\omega(a,c,b)$;
\item
\emph{Transitive}, i.e. $\omega(a,b,c)=1$ and $\omega(a,c,d)=1$ implies $\omega(a,b,d)=1$;
\item
\emph{Total}, i.e. $\omega(a,b,c)=\pm 1$ whenever $a,b,c$ are distinct.
\end{enumerate}

Intuitively, there is a left invariant choice of orientation on every (cyclically) ordered triple in $G$.
Note that a left orderable group is automatically cyclically orderable. If $a< b< c$ we define $\omega(a,b,c)=1$, and the remaining values of $\omega$ are determined by this condition.

The following result is oftentimes called the \emph{dynamical realization} of an order, and was for a long time a folklore theorem. In the statement, we denote the closed unit interval by $I$.

\begin{thm}[See~\cite{DeroinNavasRivas,Navas2011}, for instance]\label{thm:dyn realization}
Let $G$ be a countable group.
\begin{enumerate}
\item
The group $G$ injects into $\Homeo^+(\R)$ and $\Homeo^+(I)$ if and only if $G$ is left orderable;
\item
The group $G$ injects in $\Homeo^+(S^1)$ if and only if $G$ is cyclically orderable.
\end{enumerate}
\end{thm}

From Theorem~\ref{thm:dyn realization}, it is again easy to see that a (countable) left orderable group is automatically cyclically orderable. If $\phi\in\Homeo^+(\R)$, we compactify $\R$ with a point at infinity, and extend $\phi$ to $S^1$ by declaring it to fix the point at infinity.

\subsection{Implicit homeomorphism actions of right-angled Artin groups}

Theorem \ref{thm:dyn realization} says that the moment we find a left invariant ordering on a countable group $G$, then we have implicitly found an action of $G$ by homeomorphisms on $\R$. Here, we show how to produce many left orderings of a right-angled Artin groups. Recall that in Subsection~\ref{sssec:criterion}, we mentioned the fact that right-angled Artin groups are residually torsion--free nilpotent.

\begin{prop}[See~\cite{DDRW08}]\label{prop:rtfn}
Let $G$ be a finitely generated, residually torsion--free nilpotent group. Then $G$ is left orderable.
\end{prop}

In the proof we will show that $G$ is in fact bi--orderable, i.e. admits a total order which is left invariant and conjugacy invariant.

\begin{proof}[Proof of Proposition~\ref{prop:rtfn}]
Choose a sequence of nested normal subgroups $\{\gamma_i(G)\}_{i\geq 1}$ of $G$ such that $\gamma_1(G)=G$, such that \[\bigcap_{i\geq 1}\gamma_i(G)=1,\] such that $G/\gamma_i(G)$ is torsion--free nilpotent for each $i$, and such that \[\gamma_{i-1}(G)/\gamma_i(G)<Z(G/\gamma_i(G)).\] Note that $\gamma_{i-1}(G)/\gamma_i(G)$ is a finitely generated, torsion--free abelian group for each $i\geq 2$, and it is easy to see that any such group admits a bi--invariant ordering. For each $i$, we choose such an ordering arbitrarily. Now if $g,h\in G$ are distinct, choose $i$ minimal such that $1\neq g^{-1}h\in G/\gamma_i(G)$. We declare $g<h$ if and only if $1<g^{-1}h\in\gamma_{i-1}(G)/\gamma_i(G)$. It is straightforward to check that this is a bi--invariant ordering on $G$.
\end{proof}

\subsection{Explicit homeomorphism actions of right-angled Artin groups}

In this subsection, we produce an explicit embedding $A(\gam)\to\Homeo^+(\R)$ which will be illustrative for our discussion in the sequel. The construction we give here can be found in~\cite{BKK2014}. We begin with some facts about the combinatorial group theory of right-angled Artin groups.

Let $1\neq w\in A(\gam)$ be a reduced word. We say that $w$ is a \emph{central word} if $\supp w$ is a complete subgraph of $A(\gam)$. Since singleton vertices are complete subgraphs of $\gam$, any element of $A(\gam)$ can be written as a product of central words. Now, let \[w=w_k\cdots w_1\] be a (reduced) product of central words. We say that $w$ is in \emph{left greedy normal form} if for each $i<k$ and each $v\in\supp w_i$, there exists a $v'\in\supp w_{i+1}$ such that $[v,v']\neq 1$.

\begin{prop}[cf.~\cite{Koberda2012}]\label{prop:leftgreedyexist}
Every reduced element $1\neq w\in A(\gam)$ can be written in left greedy normal form.
\end{prop}
\begin{proof}[Sketch of proof]
Write $w=w_k\cdots w_1$ as a product of central words. For each $v\in\supp w_1$, find a vertex $v'\in\supp w_2$ such that $[v,v']\neq 1$. If no such vertex exists, move all the occurrences of $v$ (or $v^{-1}$) in $w_1$ into $w_2$. Repeat this process for each $w_i$ as many times as is necessary until no more vertices can be moved to the left. If $w_i$ is empty at any point in this process, we simply delete it from the expression of $w$. Since vertices are always moved only into central words with higher indices and since the index $k$ cannot increase, this process will terminate after finitely many steps.
\end{proof}

The following proposition is an immediate corollary of the definition of left greedy normal form.

\begin{prop}\label{prop:leftgreedy}
Let $1\neq w_k\cdots w_1\in A(\gam)$ be in left greedy normal form. Then for each $i$ there exists a vertex $v_i\in\supp w_i$ such that for each $i<k$ we have $[v_i,v_{i+1}]\neq 1$.
\end{prop}

\begin{thm}\label{thm:raaghomeo}
Let $\gam$ be a finite simplicial graph. Then there exists an injective map $A(\gam)\to\Homeo^+(\R)$.
\end{thm}

The proof will establish the following more precise statement: the group $A(\gam)$ is residually $\Homeo^+(I)$, i.e. for each $1\neq w\in A(\gam)$ there exists a homomorphism $\phi\colon A(\gam)\to\Homeo^+(I)$ such that $w\notin\ker\phi$. Since $A(\gam)$ is countable and since $\R$ can be written as a countable union of copies of $I$, this will suffice to prove Theorem~\ref{thm:raaghomeo}.

\begin{proof}[Proof of Theorem~\ref{thm:raaghomeo}]
Let $1\neq w\in A(\gam)$. By Proposition~\ref{prop:leftgreedyexist}, we can write $w=w_k\cdots w_1$ in left greedy normal form. Choose $v_i\in \supp w_i$ satisfying Proposition~\ref{prop:leftgreedy}. We define an action of $A(\gam)$ on $I$ as follows: first, choose a chain of $k$ open subintervals $\{J_1,\ldots, J_k\}$ of $I$ such that $J_i\cap J_{i+1}$ is a nonempty proper subinterval of $J_i$ and of $J_{i+1}$ for $i<k$, and such that $J_i$ and $J_{\ell}$ are disjoint otherwise. Moreover, we require these intervals to be increasing, i.e. the left and right endpoints of $J_i$ lie to the left of the left and right endpoints of $J_{i+1}$, respectively. Choose a point $x_1\in J_1\setminus J_2$ and increasing homeomorphisms $\phi_i$ supported exactly on $J_i$ satisfying the following: inductively define $x_{i+1}=\phi_i(x_i)$, we require $x_{i+1}\in J_{i+1}$, and $x_{k+1}\notin J_{k-1}$.

Let $\epsilon_i$ be the sign of $v_i$ as it occurs in $w_i$. We define the action of $v_i$ on $J_i$ to be via $\phi_i^{\epsilon_i}$. Then, we set the action of $v\in V(\gam)$ on $I$ to be by \[\prod_{v_i=v}\phi_i^{\epsilon_i}.\] Observe that this is a well--defined action of $A(\gam)$ on $I$. Moreover, it is clear that $x_{k+1}=w(x_1)\in J_k\setminus J_{k-1}$. Thus, if $k=1$ then there is nothing to check, and if $k>1$ then $x_{k+1}\neq x_1$, so that $w$ acts nontrivially on $I$.
\end{proof}

\section{Groups of $C^1$ diffeomorphisms}

We now turn our attention to actions of groups with regularity beyond mere continuity. Differentiability of a group action does impose nontrivial restrictions on groups which can act on a given one--manifold, though there are relatively few systematic approaches to ruling out the action of a particular group. One of the strongest tools in studying $C^1$ actions on one--manifolds is called Thurston Stability, which finds its origins in foliation theory. Let $G$ be a group acting on a topological space $X$, and let $x\in X$ be a global fixed point of $G$. We write $G_x$ for the group of \emph{germs} of $G$ at $x$. Namely, we consider the group $G$ modulo the equivalence relation $g\sim h$ if $g$ and $h$ agree on a neighborhood of $x$. One statement of Thurston Stability is the following:

\begin{thm}[Thurston Stability,~\cite{ThurstonStability}]
Let $G<\Diff^1(M)$ be finitely generated, where $M$ is a one--manifold, and let $x$ be a global fixed point of $G$. Then if $G_x$ is nontrivial we have $H^1(G_x,\R)\neq 0$.
\end{thm}
\begin{proof}[Sketch of proof, following~\cite{Calegari2007}]
Assume $G_x$ is nontrivial. Without loss of generality, we may assume $x=0$. First, one may consider the homomorphism $G\to\R$ given by $\rho\mapsto\log\rho'(0)$. If this homomorphism is nontrivial then we are done, so we may assume that $G$ lies in the kernel. One then rescales the action of $G$ suitably and extracts a limiting action using Ascoli's Theorem, obtaining an action of $G_x$ by translations. H\"older's Theorem then implies that $G_x$ admits a nontrivial homomorphism to $\R$.
\end{proof}

In particular, $G$ is either trivial or surjects to $\Z$. D. Calegari was the first to produce an explicit finitely presented group of $\Homeo^+(S^1)$ which admits no faithful homomorphism to $\Diff^1(S^1)$ in~\cite{MR2218985}. In his example, he leveraged Thurston Stability by taking a certain perfect group $H$ (i.e. $H=[H,H]$)  which is naturally realized as the fundamental group of a certain $2$--dimensional hyperbolic orbifold and hence admits a faithful analytic action on the circle, and embedding it in a larger group $G$. He then shows that whenever $G$ acts faithfully on the circle, there is a point whose stabilizer is isomorphic to $H$, whence Thurston Stability precludes differentiability.

Thurston Stability says nothing about right-angled Artin groups acting by $C^1$ diffeomorphisms, since right-angled Artin groups are \emph{locally indicable}, i.e. every finitely generated subgroup surjects to $\Z$. Locally indicable groups, however, may not admit faithful $C^1$ actions on one--manifolds, as was first shown by A. Navas~\cite{Navas2010}.

It turns out that if $M$ is any one--manifold and $\gam$ is any finite simplicial graph, then $A(\gam)<\Diff^1(M)$. In fact, a stronger result holds:

\begin{thm}[Farb--Franks~\cite{FF2003}, Jorquera~\cite{Jorquera}]\label{thm:rtfn}
Let $G$ be a finitely generated residually torsion--free nilpotent group and let $M$ be a one--manifold. Then $G<\Diff^1(M)$.
\end{thm}

Since right-angled Artin groups are manifestly finitely generated and residually torsion--free nilpotent, we obtain the following immediate corollary:

\begin{cor}\label{cor:raagc1}
Let $\gam$ be a finite simplicial graph and let $M$ be a one--manifold. Then $A(\gam)<\Diff^1(M)$.
\end{cor}

The proof of Theorem~\ref{thm:rtfn} requires some rather subtle analytic estimates, so we will only outline the basic idea in a special case. Let $\lambda$ be a positive real number and let $N$ be a given finitely generated torsion--free nilpotent group. For simplicity, we will assume that the lower central series of $N$ has length two, so that $[N,N]<Z(N)$. The goal is to produce a faithful action of $N$ on an interval of length $\lambda$ so that the generators of $N$ act by $C^1$ diffeomorphisms which are tangent to the identity at the endpoints of the interval. One carefully chooses one homeomorphism for each generator of $N$ which is fully supported on the interval and tangent to the identity at the endpoints, so that the commutator subgroup of $N$ is supported on a countable union of subintervals. With care, one can arrange for the commutator subgroup to act freely on each interval of its support so that it is abelian by H\"older's Theorem (see Theorem~\ref{thm:holder}), and one can arrange the commutator subgroup to be periodic with respect to the conjugation action of the generators of $N$. This way, the commutator subgroup will commute with the conjugation action of $N$ and hence lie in the center of $N$.

\section{Groups of $C^2$ diffeomorphisms}

Increasing the regularity of a group action to twice--differentiability, many fewer groups can act on one--manifolds. One of the most powerful tools in this situation is Kopell's Lemma~\cite{Kopell1970}, which can be stated as follows:

\begin{thm}\label{thm:kopell}
Let $f,g\in\Diff^2(I)$ be such that $[f,g]=1$. Suppose $\Fix f\cap (0,1)=\varnothing$ and $\Fix g\cap (0,1)\neq\varnothing$. Then $g=1$.
\end{thm}

A particularly efficient proof of Theorem~\ref{thm:kopell} is given by Calegari~\cite{Calegari2007}, which we reproduce here for the convenience of the reader. First, we require the following fact which is very easy but nevertheless useful:

\begin{lem}\label{lem:fix}
Let $X$ be a set and let $f,g\in\Sym(X)$ be such that $[f,g]=1$. Then $f(\Fix g)\subset\Fix g$.
\end{lem}
\begin{proof}
Let $x\in Fix(g)$. Then $f(x)=f(g(x))=g(f(x))$, so that $f(x)\in\Fix g$.
\end{proof}

\begin{proof}[Proof of Theorem~\ref{thm:kopell}]
Note that $f$ acts on $(0,1)\cong\R$, and since $f$ is fixed point free, it is topologically conjugate to translation by $1$ by a homeomorphism $h\colon (0,1)\to\R$. A \emph{fundamental domain} for the action of $f$ will be defined to be the subinterval $J=(h^{-1}(0),h^{-1}(1))$. By definition, $f^n(J)\cap J=\varnothing$ for any $n\neq 0$.

The assumption that $f$ is $C^2$ implies that the function $\log|f'|$ is Lipschitz and is hence of bounded variation on $I$. Choosing $a,b\in J$, this implies that \[\sum_{i=-\infty}^{\infty}\bigg\vert\log|f'(f^i(a))|-\log|f'(f^i(b))|\bigg\vert\leq K\] for some constant that is independent of both $a$ and $b$. From this estimate and the chain rule, there exists a positive constant $C$ such that \[\frac{1}{C}\leq \Bigg\vert\frac{(f^n)'(a)}{(f^n)'(b)}\Bigg\vert\leq C,\] independently of $n\in \Z$. From the chain rule and the fact that $\phi=g^m$ commutes with $f$, we obtain \[\frac{(f^n)'(x)}{(f^n)'(\phi(x))}=\frac{\phi'(x)}{\phi'(f^n(x))}\] for every $n\in\Z$.

Suppose for a contradiction that $g$ is not the identity. Since $g$ has at least one fixed point, we may assume that $g$ has a fixed point $p$ such that for some $x$ near $p$ we have $g^k(x)\to p$ as $k\to\infty$. Note that if $x$ was chosen sufficiently close to $p$, we have that the expression \[\Bigg\vert\frac{(f^n)'(x)}{(f^n)'(\phi(x))}\Bigg\vert\] is bounded away from zero.

Since $p$ is fixed, we can fix $\epsilon>0$ and find an $m$ such that \[\frac{|g^m(x)-p|}{|x-p|}<\epsilon.\] In particular, for any $\delta>0$, we can choose an $m>0$ such that \[|\phi'(x)|=|(g^m)'(x)|<\delta.\]

Now, applying Lemma~\ref{lem:fix}, we have that fixed points of $g$ accumulate at $0$. Since $g$ is $C^1$, it follows that $\phi'(0)=1$. Thus for a fixed $\delta>0$, there is an $n$ such that $|\phi'(f^n(x))|\geq 1-\delta$. But then we have the estimate \[\Bigg\vert\frac{\phi'(x)}{\phi'(f^n(x))}\Bigg\vert\leq\frac{\delta}{1-\delta}.\] Since $\delta$ was arbitrary, this is a contradiction.
\end{proof}

\subsection{Nilpotent groups and the Plante--Thurston Theorem}

By Theorem~\ref{cor:raagc1}, we have that every right-angled Artin group acts faithfully by $C^1$ diffeomorphisms on every one--manifold. We first show that the methods of Theorem~\ref{thm:rtfn} do not generalize to $C^2$ diffeomorphisms. This is the content of the Plante--Thurston Theorem and its generalization to $\R$ due to Farb--Franks.

\begin{thm}[Plante--Thurston~\cite{PT1976}, cf.~\cite{Navas2011}]\label{thm:ptcompact}
Let $N<\Diff^2(M)$ be a nilpotent subgroup, where $M$ is a compact, connected one--manifold. Then $N$ is abelian.
\end{thm}
\begin{proof}
We will prove the result in the special case where $M=I$. We claim that for each $1\neq f\in N$, we have $\Fix f\cap (0,1)=\varnothing$. This will prove the result, since Theorem~\ref{thm:holder} then implies that $N$ is abelian.

We assume that $N$ has no global fixed points in $I$, and let $g\in Z(N)$ be a central element. Suppose that $1\neq f\in N$ has at least one fixed point $x\in (0,1)$. Assume that $x\in\partial\Fix f$. We claim that in this case, $g(x)=x$. Otherwise, there is an interval $[a,b]$ with nonempty interior such that \[\{a,b\}=\{\lim_{n\to\pm\infty}g^{n}(x)\}.\] Since $g$ and $f$ commute, Lemma~\ref{lem:fix} implies that $f$ fixes both $a$ and $b$. Kopell's Lemma (Theorem~\ref{thm:kopell}) now implies that if $g(x)\neq x$ then $f$ is the identity on $[a,b]$. Since we assume $x\in\partial\Fix f$ and since $x\in(a,b)$, this is a contradiction and we conclude that $g(x)=x$.

Let $y\in\Fix g\cap (0,1)$ be a boundary point of the fixed point set of $g$. Since $g$ is in the center $N$, we have that $h(y)=y$ for all $h\in N$, by the same argument as before. It follows that $y$ is a global fixed point of $N$, a contradiction.
\end{proof}

The following generalization of Theorem~\ref{thm:ptcompact} holds in the noncompact setting:

\begin{thm}[See~\cite{FF2003}]\label{thm:ptr}
Let $N<\Diff^2(\R)$ be a nilpotent subgroup. Then $N$ is metabelian, i.e. $[N,N]$ is abelian.
\end{thm}
\begin{proof}[Sketch of proof]
First, a variation of the argument in the proof of Theorem~\ref{thm:ptcompact} can be used to prove that if every element of a nilpotent group action on $\R$ has a fixed point then the group must be abelian. Then, one takes an arbitrary nilpotent group $N$ acting on $\R$ and one finds an $N$--invariant measure on $\R$ which is finite on compact sets. This furnishes a homomorphism $N\to\R$ given by looking at the translation number. Each element of the kernel of this map has a fixed point, whence $N$ must be metabelian.
\end{proof}

In~\cite{FF2003}, Farb--Franks prove that whereas not every finitely generated nilpotent group occurs as a subgroup of $\Diff^2(\R)$, there are nilpotent subgroups of arbitrary lower central series length.

\subsection{$C^2$ actions of right-angled Artin groups}

Theorems~\ref{thm:ptcompact} and~\ref{thm:ptr} show that one cannot leverage nilpotent groups to find right-angled Artin subgroups of $\Diff^2(M)$ for any one--manifold. We will see below in Theorem~\ref{thm:raagr} that every right-angled Artin group does indeed embed into $\Diff^2(\R)$, and in fact into $\Diff^{\infty}(\R)$. This suggests that right-angled Artin groups should occur plentifully in $\Diff^2(M)$ when $M$ is compact. The opposite is true, however.

Let $P_4$ be the graph which is a path on four vertices, which is depicted in Figure~\ref{f:p4}.

\begin{figure}[h!]
  \tikzstyle {a}=[red,postaction=decorate,decoration={%
    markings,%
    mark=at position .5 with {\arrow[red]{stealth};}}]
  \tikzstyle {b}=[blue,postaction=decorate,decoration={%
    markings,%
    mark=at position .43 with {\arrow[blue]{stealth};},%
    mark=at position .57 with {\arrow[blue]{stealth};}}]
  \tikzstyle {v}=[draw,shape=circle,fill=black,inner sep=0pt]
  \tikzstyle {bv}=[black,draw,shape=circle,fill=black,inner sep=1pt]
  \tikzstyle{every edge}=[-,draw]
\begin{tikzpicture}[thick]
\draw (-1,0) node [bv] {} node [above=.1] {\small $a$} 
-- (0,0)  node [bv] {} node [above=.1] {\small $b$} 
-- (1,0)  node [bv] {} node [above=.1] {\small $c$} 
-- (2,0)  node [bv] {} node [above=.1] {\small $d$};
\end{tikzpicture}%
\caption{The graph $P_4$.}
\label{f:p4}
\end{figure}
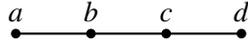

The corresponding right--angled Artin group has the presentation \[A(P_4)\cong\langle a,b,c,d\mid [a,b]=[b,c]=[c,d]=1\rangle.\]

The group $A(P_4)$ is a specific right-angled Artin group, but it is privileged in the sense that it is pervasive among right-angled Artin groups which are sufficiently complicated. We say that a finite simplicial graph $\gam$ is a \emph{cograph} if $P_4\nsubset\gam$. It turns out that right-angled Artin groups on cographs are exactly the ones which do not contain a copy of $A(P_4)$. Moreover, these groups admit a straightforward characterization.

\begin{prop}[See~\cite{KK2013}, for instance]\label{prop:cograph}
Let $\gam$ be a finite simplicial graph. Then $\gam$ is a cograph if and only if $A(P_4)$ is not a subgroup of $A(\gam)$. Moreover, the class $\mathcal{K}$ of right-angled Artin groups on cographs is characterized by the following:
\begin{enumerate}
\item
The group $\Z\in\mathcal{K}$;
\item
The class $\mathcal{K}$ is closed under taking finite direct products;
\item
The class $\mathcal{K}$ is closed under taking finite free products.
\end{enumerate}
\end{prop}

When $M$ is compact, the right-angled Artin subgroups of $\Diff^2(M)$ are very limited:

\begin{thm}[See~\cite{BKK16}]\label{thm:p4}
Let $M$ be a compact one--manifold. Then the group $A(P_4)$ is not a subgroup of $\Diff^2(M)$. In particular, if $A(\gam)<\Diff^2(M)$ then $\gam$ is a cograph.
\end{thm}

The proof of Theorem~\ref{thm:p4} is rather involved and we will not give it here, though we will sketch some of the main ideas. The idea is to translate the combinatorial structure of the group $A(P_4)$ into a dynamical setup of subintervals and diffeomorphisms acting on $M$, all the time making liberal use of Kopell's Lemma and the Mean Value Theorem.

The first thing one does to prove Theorem~\ref{thm:p4} is to use a putative faithful $C^2$ action $\phi\colon A(P_4)\to \Diff^2(M)$ to force the supports of the generators \[\{\phi(a),\phi(b),\phi(c),\phi(d)\}\] to exhibit certain desirable properties. Namely, consider the chains of intervals as in Figure~\ref{f:coint}, where the intervals $I_v$ are components of $\supp\phi(v)$. Observe that restricting to such a configuration, the commutation relations defining $A(P_4)$ are satisfied by the corresponding diffeomorphisms.

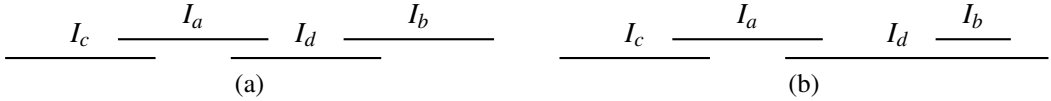
\begin{figure}[h!]
  \tikzstyle {a}=[black,postaction=decorate,decoration={%
    markings,%
    mark=at position 1 with {\arrow[black]{stealth};}    }]
  \tikzstyle {bv}=[black,draw,shape=circle,fill=black,inner sep=1.5pt]
\subfloat[(a)]{\begin{tikzpicture}[thick,scale=.5]
\draw (-5,0) -- (-1,0);
\draw (-3,0) node [above] {\small $I_c$}; 

\draw (-2,.5) -- (2,.5);
\draw (0,.5) node [above] {\small $I_a$}; 

\draw (1,0) -- (5,0);
\draw (3,0) node [above] {\small $I_d$}; 

\draw (4,.5) -- (8,.5);
\draw (6,.5) node [above] {\small $I_b$}; 
\end{tikzpicture}}
\quad\quad
\subfloat[(b)]{\begin{tikzpicture}[thick,scale=.5]
\draw (-5,0) -- (-1,0);
\draw (-3,0) node [above] {\small $I_c$}; 

\draw (-2,.5) -- (2,.5);
\draw (0,.5) node [above] {\small $I_a$}; 

\draw (1,0) -- (8,0);
\draw (4,0) node [above] {\small $I_d$}; 

\draw (5,.5) -- (7,.5);
\draw (6,.5) node [above] {\small $I_b$}; 
\end{tikzpicture}}
\caption{Chains of intervals.}
\label{f:coint}
\end{figure}

The most difficult part of the proof of Theorem~\ref{thm:p4} is to show that if the action of $A(P_4)$ on $M$ is faithful, then there must exist infinitely many configurations of intervals $\{\{I_a,I_b,I_c,I_d\}_n\}_{n\geq 1}$ in the supports of the generators $\{\phi(a),\phi(b),\phi(c),\phi(d)\}$ which have the same intersection pattern (up to orientation) as the intervals in Figure~\ref{f:coint}. Moreover, if $x$ denotes the right endpoint of $I_c$, for infinitely many of these configurations, we must have $\phi(d)\circ\phi(a)(x)\in I_b$. This part of the proof requires both Kopell's Lemma and the faithfulness of the action of $A(P_4)$.

Since $M$ is compact, it follows that these configurations $\{\{I_a,I_b,I_c,I_d\}_n\}_{n\geq 1}$ must get arbitrarily small as $n$ tends to infinity. One then uses the Mean Value Theorem to deduce that at least one of $\phi(a)$ or $\phi(d)$ fails to be $C^1$, which establishes the result. Note that this does not contradict Corollary~\ref{cor:raagc1}, since Kopell's Lemma was used in an essential way.

Theorem~\ref{thm:p4} makes significant progress towards but does not quite complete the classification of right-angled Artin subgroups of $\Diff^2(M)$. The following question remains open:

\begin{que}\label{que:cograph}
Let $\gam$ be a cograph and let $M$ be compact. Do we have $A(\gam)<\Diff^2(M)$?
\end{que}

The simplest case of Question~\ref{que:cograph} which is unknown is the case where $\gam$ is a path on three vertices together with an isolated vertex, i.e. $A(\gam)\cong (F_2\times\Z)*\Z$. A more general question about the algebraic structure of diffeomorphism groups which is unknown is the following, a positive answer to which would give a positive answer to Question~\ref{que:cograph}:

\begin{que}\label{que:freeprod}
Let $\mathcal{G}$ denote the class of finitely generated subgroups of $\Diff^2(M)$, where $M$ is a compact one--manifold. Is $\mathcal{G}$ closed under finite free products?
\end{que}

It seems to the author that the case of $(F_2\times\Z)*\Z$ contains all the salient features of the general case of Question~\ref{que:freeprod}. On the one hand, if the answer to Question~\ref{que:cograph} for this group is no then we would obtain a new proof of Theorem~\ref{thm:p4} since $(F_2\times\Z)*\Z<A(P_4)$. On the other hand, if the answer to Question~\ref{que:cograph} is yes then we would obtain a dynamical characterization of cographs.

\begin{rem}
At the time of the original writing, Questions~\ref{que:cograph} and~\ref{que:freeprod} were open. They have both been very recently resolved in the negative by Kim and the author~\cite{KKFreeProd2017}.
\end{rem}

\section{Non--compactness and $C^{\infty}$ diffeomorphisms}

If one dispenses with the assumption that the manifold $M$ on which the right-angled Artin group acts is compact, the algebraic restrictions on right-angled Artin groups with highly regular actions on one--manifolds disappear.

\subsection{$C^{\infty}$ actions}
Theorem~\ref{thm:raaghomeo} contains all the necessary ideas to prove the following result, which is the primary result in the paper~\cite{BKK2014} of Baik, Kim, and the author:

\begin{thm}\label{thm:raagr}
Let $\gam$ be an arbitrary right-angled Artin group. Then there exists an injective $A(\gam)\to\Diff^{\infty}(\R)$.
\end{thm}
\begin{proof}[Sketch of proof]
To establish the result, we can mimic the proof of Theorem~\ref{thm:raaghomeo} above. It suffices to exercise sufficient care to ensure that all the choices of homeomorphisms are $C^{\infty}$. If $1\neq w\in A(\gam)$ has length $n$, it is fairly easy to use the construction of Theorem~\ref{thm:raaghomeo} to produce an action of $A(\gam)$ on the interval $[0,n]$ so that the derivatives of the generators are universally bounded. This way, one obtains an embedding $A(\gam)\to\Diff^{\infty}(\R)$.
\end{proof}

\subsection{Analytic actions}

The highest degree of regularity one can consider on a compact one--manifold is analyticity, and we write $\Diff^{\omega}(M)$ for the group of analytic diffeomorphisms of $M$. Recall that if $t$ is a local parameter of $M$, an analytic diffeomorphism of $M$ is given by a locally convergent power series in $t$.

The principle of analytic continuation severely restricts commutation of diffeomorphisms in $\Diff^{\omega}(M)$, and hence the diversity of right-angled Artin subgroups of $\Diff^{\omega}(M)$. The following is a standard result about analytic functions, which applied more generally than as we state here:

\begin{lem}\label{lem:analytic}
Let $M$ be a one--manifold and let $f$ and $g$ be analytic functions on $M$. Let $\{x_i\}_{i\geq 1}\subset M$ be a sequence of points with accumulation point $x\in M$, such that $f(x_i)=g(x_i)$ for each $i$. Then $f=g$ on $M$.
\end{lem}
\begin{proof}
By the continuity of all derivatives, we have that the value and all derivatives of $f-g$ at $x$ are $0$, whence the power series expansion of $f-g$ is identically $0$.
\end{proof}

Applying Lemma \ref{lem:analytic} to the identity function, we obtain the following:

\begin{cor}\label{cor:analfix}
Let $M$ be a compact one--manifold and let $f\in\Diff^{\omega}(M)$. Then $f$ has at most finitely many fixed points.
\end{cor}

Combining Corollary~\ref{cor:analfix} with Lemma~\ref{lem:fix}, we see the following:

\begin{cor}\label{cor:commufix}
Let $M$ be a compact one--manifold and let $f,g\in\Diff^{\omega}(M)$ be such that $[f,g]=1$. Suppose moreover that both $f$ and $g$ have infinite order. Then for some $N,M\gg 0$ we have $\Fix(f^N)=\Fix(g^M)$.
\end{cor}
\begin{proof}
The only potential complication here is when $M$ has a component homeomorphic to $S^1$ and one of $f$ or $g$ has periodic points. Replacing $f$ and $g$ by suitable powers, we may assume that both are fixed point free (in which case the conclusion of the corollary holds trivially), or $f$ has a finite collection of fixed points and $g$ has no periodic points. Lemma~\ref{lem:fix} implies that $g(\Fix f)=\Fix f$, so that $\Fix f$ is pointwise stabilized by $g$. Switching the roles of $f$ and $g$, we obtain the desired conclusion.
\end{proof}

Thus, a nontrivial analytic homeomorphism $f$ is \emph{fully supported} in the sense that the interior of $\Fix f$ is empty, and two infinite order commuting analytic homeomorphisms of a compact one--manifold have the same fixed points, possibly after passing to a power.

Let $G$ be a group and let $S\subset G$. The \emph{commutation graph} of $S$ is the graph whose vertices are the elements of $S$ and whose edge relation is given by commutation in $G$.

The following lemma also holds for $C^2$ diffeomorphisms (see~\cite{FF2001}), but for simplicity we state it for analytic diffeomorphisms.
\begin{lem}\label{lem:commgraph}
Let $M$ be a compact one--manifold and let $S\subset\Diff^{\omega}(M)\setminus 1$ be a collection of diffeomorphisms. Suppose that the commutation graph of $S$ is connected. Suppose moreover that each $g\in S$ has at least one fixed point on each component of $M$ which is homeomorphic to $S^1$. Then $G=\langle S\rangle$ is abelian.
\end{lem}
\begin{proof}
By Corollary~\ref{cor:commufix}, the set of fixed points of any element of $S$ are global fixed points of $G$, so we may consider the case where $G$ acts without global fixed points on a single interval $I$. Let $g\in S$ and suppose that there is an element $h\in Z_G(g)$ with a fixed point in the interior of $I$. Then the fact that $h$ is $C^2$ implies that $h$ is the identity, by Kopell's Lemma. Thus, $Z_G(g)$ acts freely on $I$ and is hence abelian by H\"older's Theorem. Thus if $g_1,g_2\in S$, induction on distance in the commutation graph of $S$ implies that $g_1$ and $g_2$ commute, hence $G$ is abelian.
\end{proof}

Combining the facts we have gathered we obtain the following:

\begin{thm}\label{thm:analcompact}
Let $A(\gam)<\Diff^{\omega}(S^1)$. Then every connected component of $\gam$ is complete. Moreover, every such right-angled Artin group embeds in $\Diff^{\omega}(S^1)$.
\end{thm}

Theorem~\ref{thm:analcompact} was originally proved by A. Akhmedov and M. Cohen~\cite{AkhmedovCohen2015}, and we sketch a proof here:

\begin{proof}[Sketch of proof of Theorem~\ref{thm:analcompact}]
First, suppose that $A(\gam)<\Diff^{\omega}(S^1)$. Notice that $\gam$ is the commutation graph of $V(\gam)$, viewed as a subset of $A(\gam)$. Replacing generators of $A(\gam)$ by suitable powers, we may assume that the hypotheses of Lemma~\ref{lem:commgraph} are satisfied by elements of $V(\gam)$. It follows that each component of $\gam$ is complete.

Now, let $\gam$ be a graph, each component of which is complete. If $n$ is the size of the largest component of $\gam$, it suffices to show $\Z*\Z^n<\Diff^{\omega}(S^1)$, by a standard application of the Kurosh Subgroup Theorem. One can show that $\Z*\Z^n$ can be realized as a subgroup of $\PSL_2(\R)<\Diff^{\omega}(S^1)$. To see this, consider two hyperbolic axes $\gamma$ and $\delta$ which share no endpoints at infinity. For each nontrivial $w\in\Z*\Z^n$, the points on the $\PSL_2(\R)$ representation variety of $\Z*\Z^n$ for which $w$ is not in the kernel form a proper subvariety, by an easy Ping--Pong Lemma argument (see~\cite{MR1786869,KoberdaIMS}). A standard Baire Category argument proves that $\Z*\Z^n<\PSL_2(\R)$.
\end{proof}

\section{Ubiquity, mapping class groups, and applications}\label{sec:mcg}

In this section, we discuss classes of groups  which contain many right-angled Artin groups.

\subsection{Mapping class groups}\label{subsec:mcg}

Let $S$ be an orientable surface with finite genus $g$ and a finite number $n$ of punctures, marked points, and boundary components. Recall that the mapping class group $\Mod(S)$ is defined to be the group of homotopy classes of orientation preserving homeomorphisms of $S$.

An important tool for studying mapping class groups if the \emph{curve graph} of $S$, denoted $\CC(S)$. This is the graph whose vertices are homotopy classes of essential embedded loops in $S$ which are not parallel to a puncture, marked point, or boundary component, and whose edge relation is given by disjointness (i.e. two loops are adjacent if they have representatives in their homotopy classes which are disjoint on $S$). A good measure of the complexity of $\CC(S)$ is the the number \[c(S):=3g-3+n,\] which coincides with the size of the largest complete subgraph of $\CC(S)$. In the literature on mapping class groups, authors often consider the \emph{curve complex}, which is just the flag complex of the curve graph. Note that the mapping class group acts on $\CC(S)$, more or less by definition, and the quotient is finite.

The curve graph provides a good bridge between mapping class groups and right-angled Artin groups, in light of the following result.

\begin{thm}[\cite{Koberda2012}, cf. Theorem~\ref{thm:gex}]\label{thm:raagmcg}
Let $\gam$ be a finite simplicial graph and let $\i\colon \gam\to\CC(S)$ be an injective map of graphs which is adjacency (and non--adjacency) preserving. Then for all $N\gg 0$ the map \[i_N\colon A(\gam)\to\Mod(S)\] given by $v\mapsto T_{i(v)}^N$ is injective, where here $T_{i(v)}$ denotes the Dehn twist about the curve $i(v)$.
\end{thm}

Theorem~\ref{thm:raagmcg} holds in somewhat greater generality, though we shall not require such a statement here. The proof of Theorem~\ref{thm:raagmcg}, while relying only on relatively elementary hyperbolic geometry and combinatorial group theory, is rather involved and we will not give it here. Converses to Theorem~\ref{thm:raagmcg} in the sense of Theorem~\ref{thm:gex} have been studied in~\cite{KK2014b,KKOsaka2016} (cf.~\cite{BCG2016}).

Theorem~\ref{thm:raagmcg} is not the first result about right-angled Artin subgroups of mapping class groups. Interest in the subgroup structure of mapping class groups has led many authors to study right-angled Artin subgroups of mapping class groups, including Crisp--Farb~\cite{CFPreprint}, Crisp--Paris~\cite{CP2001}, Crisp--Wiest~\cite{CW2007}, Clay--Leininger--Mangahas~\cite{CLM2012}, and Kuno~\cite{Kuno2017}.

We note the following corollaries, which are fairly easy consequences of Theorem~\ref{thm:raagmcg}:

\begin{cor}\label{cor:mcgp4}
Let $S$ and $c(S)$ be as above.
\begin{enumerate}
\item
We have $A(P_4)<\Mod(S)$ if and only if $c(S)\geq 2$;
\item
If $c(S)<2$ then some finite index subgroup of $\Mod(S)$ is a free group or a direct product of a free group with $\Z$.
\end{enumerate}
\end{cor}

\begin{cor}\label{cor:allraags}
Let $\gam$ be given. Then there exists a surface $S$ such that $A(\gam)<\Mod(S)$.
\end{cor}

The following proposition is very easy yet rather useful:

\begin{prop}\label{prop:firaag}
Let $G$ be a group, let $H<G$ be a finite index subgroup, and let $A(\gam)<G$. Then $A(\gam)<H$.
\end{prop}
\begin{proof}
It is not difficult to see that for each $N\neq 0$, we have $\langle v^N\mid v\in V(\gam)\rangle\cong A(\gam)$. Let $n=[G:H]$. Then the permutation action of $G$ on the cosets of $H$ in $G$ furnishes a normal subgroup $K\leq H<G$ which has index at most $n!$. If $g\in G$ is any element, we have $g^{n!}\in K$, whence the claim of the proposition.
\end{proof}

Corollary~\ref{cor:mcgp4} and Proposition~\ref{prop:firaag} together show that mapping class groups (and their finite index subgroups) either contain a copy of $A(P_4)$, or they contain no interesting right-angled Artin groups at all.

\subsection{Braid groups}

Recall that the braid group $B_n$ on $n$ strands has the presentation \[B_n\cong\langle s_1,\ldots,s_{n-1}\mid s_is_{i+1}s_i=s_{i+1}s_is_{i+1},\, [s_i,s_j]=1 \textrm{ whenever } |i-j|>1\rangle.\] It is standard that the braid group is identified with the mapping class group of a disk with $n$ punctures or marked points, so that the majority of the discussion of Subsection~\ref{subsec:mcg} applies. What is not so immediate is Corollary~\ref{cor:allraags} for braid groups. Given $\gam$, any na\"ive construction of $S$ generally results in a surface with positive genus, and it is unclear that one can arrange the surface $S$ to be planar. One cannot apply Theorem~\ref{thm:raagmcg}, since not every graph can be realized inside of the curve graph of a multiply punctured disk. Corollary~\ref{cor:allraags} does hold for braid groups, however.

The starting point is the following result of Crisp--Wiest~\cite{CW2007}:

\begin{thm}\label{thm:crispwiest}
Suppose $\gam^c$ is planar. Then for some $n\gg 0$, we have $A(\gam)<B_n$.
\end{thm}
\begin{proof}[Sketch of proof]
Embed $\gam^c$ into a disk $D$. For each vertex $v\in V(\gam)$, we choose a subdisk $D_v\subset D$ so that $D_v\cap D_w=\varnothing$ precisely then $\{v,w\}\notin E(\gam^c)$. We then introduce some punctures so that each $D_v$ becomes a subsurface of negative Euler characteristic, and so that all intersections of subdisks are essential. We then choose essential embedded curves $\gamma_v\subset D_v$ which are not parallel to a puncture, and such that $\gamma_v\cap\gamma_w=\varnothing$ if and only if $D_v\cap D_w=\varnothing$. Applying Theorem~\ref{thm:raagmcg}, we get the conclusion of the result.
\end{proof}

With some more care and a slightly different choice of mapping classes, one can say something slightly stronger: one can arrange the subgroup $A(\gam)<B_n$ to be quasi--isometrically embedded.

Theorem~\ref{thm:crispwiest} suggests that one might attempt to embed a given right-angled Artin group $A(\gam)$ into a larger group $A(\Lambda)$, where $\Lambda^c$ is planar. We now sketch how one can do this.

We assume that $\gam^c$ is connected, which is a harmless assumption since it merely assumes that $A(\gam)$ does not split as a direct product of two right-angled Artin subgroups. Let $X$ be the universal cover of $\gam^c$, viewed as a simplicial $1$--complex, and write $p\colon X\to \gam^c$ for the covering map. If $T\subset X$ is a connected subgraph then $T$ is contractible and hence planar. Fix a vertex $x_0\in X$ and let $T_N$ be the $N$--neighborhood of $x_0$ in $X$. There is a natural, well--defined homomorphism \[\phi_N\colon A(\gam)\to A(T_N^c)\] given by \[\phi_N(v)=\prod_{x\in p^{-1}(v)\cap T_N} x.\] If $p^{-1}(v)\cap T_N$ is empty then we define this product to be the identity.

\begin{thm}[\cite{KK2015}]\label{thm:antitree}
If $N$ is large enough, then the map $\phi_N\colon A(\gam)\to A(T_N^c)$ is injective.
\end{thm}

Again, one can prove that $\phi_N$ is a quasi--isometric embedding, if $N\gg 0$. The proof of Theorem~\ref{thm:antitree} relies on cancellation theory in right-angled Artin groups, and we will not reproduce it here. Combining Theorem~\ref{thm:antitree} with Theorem~\ref{thm:crispwiest}, we obtain the following:

\begin{thm}\label{thm:raagbraid}
Let $\gam$ be a finite simplicial graph. Then for all $n\gg 0$, we have $A(\gam)<B_n$.
\end{thm}

In~\cite{KK2015}, Kim and the author provided rather poor bounds on $n$ in Theorem~\ref{thm:raagbraid} as a function of $|V(\gam)|$, but these were greatly improved by Lee--Lee~\cite{LeeLee2016}.

\subsection{Related groups}

In this subsection, we note several other natural classes of groups which are rich enough to contain lots of right-angled Artin groups. Mapping class groups are helpful in some situations:

\begin{prop}\label{prop:autfn}
If $n\geq 3$ then the groups $\Out(F_n)$ and $\Aut(F_n)$ contains $A(P_4)$.
\end{prop}

This can be seen be considering the mapping class group of a twice--punctured torus and a twice--punctured torus with a marked point, viewed as subgroups of $\Out(F_3)$ and $\Aut(F_3)$ respectively.

There are many other natural classes of groups which contain $A(P_4)$ but which do not contain subgroups commensurable with mapping class groups. Recall that the \emph{Torelli group} $\mathcal{I}(S)$ is the kernel of the homology representation \[\Mod(S)\to\Aut(H_1(S,\Z)).\] The \emph{Johnson kernel} $\mathcal{K}(S)<\mathcal{I}(S)$ is the group generated by Dehn twists about separating loops in $S$.

\begin{prop}\label{prop:torelli}
If the genus of $S$ is at least $3$, then the Torelli group $\mathcal{I}(S)$ and the Johnson kernel $\mathcal{K}(S)$ contain $A(P_4)$.
\end{prop}

The Torelli group and Johnson filtration fit inside of a more general filtration of $\Mod(S)$, called the \emph{Johnson filtration} (cf.~\cite{BassLubotzky1994,FM2012}. For simplicity, we will define it here for a closed surface with a marked point $p$, so that $\Mod(S,p)$ can be identified with a finite index subgroup of $\Aut(\pi_1(S,p))$.

Let $\gamma_k(S)$ be the $k^{th}$ term of the lower central series of $\pi_1(S)$, so that $\gamma_1(S)=\pi_1(S)$ and $\gamma_k(S)=[\pi_1(S),\gamma_{k-1}(S)]$. We write \[\mathcal{J}_k(S)=\ker\{\Mod(S)\to\Aut(\pi_1(S)/\gamma_k(S))\},\] and then $\{\mathcal{J}_k(S)\}_{k\geq 1}$ is the Johnson filtration of $\Mod(S)$. Since $\pi_1(S)$ is residually nilpotent, we have that \[\bigcap_k\mathcal{J}_k(S)=\{1\}.\] Note that with this definition, we have that $\mathcal{J}_1(S)=\Mod(S)$, that $\mathcal{J}_2(S)=\mathcal{I}(S)$, and that (though this is not obvious) $\mathcal{J}_3(S)=\mathcal{K}(S)$.

\begin{prop}\label{prop:johnson}
Let $k\geq 4$, and let $S$ have genus at least $5$. Then $A(P_4)<\mathcal{J}_k(S)$.
\end{prop}

Right-angled Artin groups are also linear over $\Z$ and hence occur inside of $\SL_n(\Z)$ for $n\gg 0$. The reader may consult Humphries~\cite{Humphries1994} and Hsu--Wise~\cite{HsuWise1999} for the original proofs of this fact. Another way to see this is to use a result of Davis--Januszkiewicz~\cite{DJ2000} which shows that a given right-angled Artin group embeds as a finite index subgroup a right-angled Coxeter group, and right-angled Coxeter groups are always linear over $\Z$.

\begin{prop}\label{prop:sln}
For $n\geq 8$, we have $A(P_4)<\SL_n(\Z)$.
\end{prop}

\subsection{Actions on compact one--manifolds}

Actions of mapping class groups on compact one--manifolds (especially the circle) are of interest from two perspectives, namely existence and non--existence.

The perspective of existence derives from a classical result of Nielsen (see~\cite{CB1988,MR2264130}).

\begin{thm}\label{thm:nielsen}
Let $S=S_{g,1}$ be a closed surface of genus $g\geq 2$ with one marked point. Then there is an injective homomorphism $\Mod(S)\to\Homeo^+(S^1)$. Moreover, this action has no global fixed points, and every orbit is dense.
\end{thm}
\begin{proof}[Sketch of proof]
Let $\psi\in\Mod(S)$, and lift $\psi$ to a homeomorphism $\Psi$ of $S$. Choose a preferred lift of the marked point of $S$ in the universal cover $\bH^2$ of $S$, and let $\yt{\Psi}$ be the lift of $\Psi$ to $\bH^2$ fixing this lift. Then the action (by quasi--isometries) of $\yt{\Psi}$ on $\bH^2$ extends to the boundary $\partial\bH^2\cong S^1$, and is independent of the choice of lift $\Psi$ of $\psi$, whence the result.
\end{proof}

While Nielsen's action of $\Mod(S)$ on $S^1$ is ``nice enough" analytically in the sense that the action is by \emph{quasi--symmetric} homeomorphisms, there is an essential non--differentiability to this action. For instance, one can verify directly that if $\gamma$ is a simple loop on $S$ and $z$ is an endpoint at infinity of a lift of $\gamma$ to $\bH^2$, then the action of the Dehn twist about $\gamma$ is not differentiable at $z$.

The question of whether or not Nielsen's action (or indeed any action of the mapping class group at all) can be smoothed naturally leads to the second perspective on mapping class group actions on the circle. This perspective derives from a broad analogy between mapping class groups and lattices in semisimple Lie groups. The mapping class group exhibits some behavior of a lattice in a rank one lattice, and some behavior of a higher rank lattice. Many rank one lattices have faithful actions on the circle, whereas higher rank lattices never do:

\begin{thm}[Witte--Morris~\cite{Witte1994}, Burger--Monod~\cite{BM1999}, Ghys~\cite{Ghys1999}]\label{thm:lattices1}
Let $G$ be a higher rank lattice in a non--split semisimple Lie group of rank at least two.
\begin{enumerate}
\item
Any $C^0$ action of $G$ on $S^1$ has a finite orbit;
\item
Any $C^1$ action of $G$ on $S^1$ factors through a finite group.
\end{enumerate}
\end{thm}

The second part of Theorem~\ref{thm:lattices1} can be deduced from the first. Let $p$ be a fixed point of the $C^1$ action of a finite index subgroup $G_0<G$. Then Thurston Stability implies that the germ of the action of $G_0$ at $p$ is either trivial or admits a nontrivial homomorphism to $\R$. But $G_0$, being a lattice in higher rank, has Kazhdan's Property $(T)$ (see~\cite{ValetteBook2008}) and thus admits no such homomorphisms. One then shows easily that that $G_0$ acts by the identity.

Theorem~\ref{thm:lattices1} is nicely complemented by a result of Navas:

\begin{thm}[\cite{Navas2002ASENS}]\label{thm:navasT}
Let $G$ be a countable group with Kazhdan's Property $(T)$. Then any $C^{1+\alpha}$ action of $G$ on the circle factors through a finite group, provided $\alpha>1/2$.
\end{thm}

In light of Theorem~\ref{thm:nielsen}, it is clear that no immediate generalization of Theorem~\ref{thm:lattices1} holds for mapping class groups. Increasing the regularity of the actions, we have the following results:

\begin{thm}[Farb--Franks, Ghys (see~\cite{FF2001})]\label{thm:ffghys}
Let $S$ be a surface of genus at least $3$. Then any $C^2$ action of $\Mod(S)$ on a compact one--manifold is trivial.
\end{thm}

Theorem~\ref{thm:ffghys} was generalized by Parwani~\cite{Parwani2008} to $C^1$ actions:

\begin{thm}\label{thm:parwani}
Let $S$ be a surface of genus at least $6$. Then any $C^1$ action of $\Mod(S)$ on a compact one--manifold is trivial.
\end{thm}

As might be expected, the main dynamical tool used to prove Theorem~\ref{thm:ffghys} is Kopell's Lemma, and for Theorem~\ref{thm:parwani} it is Thurston Stability. In both of these results, the essential facts about mapping class groups that are used are the fact that Dehn twists about curves intersecting once satisfy braid relations, and that sufficiently complicated mapping class groups have trivial abelianization. Kopell's Lemma can be used to show that braid relations collapse under $C^2$ actions of sufficiently complicated mapping class groups, and the lack of nontrivial homomorphisms from mapping class groups to $\R$ allows one to apply Thurston Stability.

Theorems~\ref{thm:ffghys} and~\ref{thm:parwani} require the full mapping class group, whereas a lattice in a Lie group is only defined up to commensurability. Braid relations disappear in finite index subgroups of mapping class groups, and the existence or non--existence of nontrivial homomorphisms to $\R$ for finite index subgroups of mapping class groups is a well--known open problem. However, using right-angled Artin groups (and precisely Theorem~\ref{thm:p4}), one can prove analogues of Theorem~\ref{thm:lattices1} for mapping class groups and other groups:

\begin{thm}\label{thm:virtualmcg}
Let $M$ be a compact one--manifold. No finite index subgroup of the following groups admits an injective homomorphism to $\Diff^2(M)$:
\begin{enumerate}
\item
The mapping class group $\Mod(S)$ for $c(S)\geq 2$;
\item
The braid group $B_n$ for $n\geq 4$;
\item
The Torelli group and Johnson kernel in genus at least $3$;
\item
Higher terms of the Johnson filtration in genus at least $5$;
\item
The groups $\Aut(F_n)$ and $\Out(F_n)$ for $n\geq 3$;
\item
The group $\SL_n(\Z)$ for $n\geq 8$.
\end{enumerate}
\end{thm}

Of course the conclusion of Theorem~\ref{thm:virtualmcg} for $\SL_n(\Z)$ is already implied by Theorem~\ref{thm:lattices1}. In fact, Witte--Morris~\cite{Witte1994} has proved that for $n\geq 3$, the group $\SL_n(\Z)$ is not virtually left orderable, so that any $C^0$ action on the circle already factors through a finite group.

The action of the entire group $\Aut(F_n)$ or $\Out(F_n)$ on the circle always factors through a finite group, as was shown in~\cite{BridsonVogtmann2003}.

As for mapping class groups and $C^0$ actions, it is easy to see that if $S$ is closed then $\Mod(S)$ generally contains noncyclic finite groups and hence admits no faithful action on the circle. If $S$ has a marked point then $\Mod(S)$ still has torsion and hence admits no faithful action on $\R$. Since mapping class groups are virtually torsion--free, these obstructions vanish after passing to a finite index subgroup. The following question is one of the most well--known open questions in orderable groups:

\begin{que}
Let $S$ be a closed surface. Does there exist a finite index subgroup $G<\Mod(S)$ and an injective homomorphism $G\to\Homeo^+(\R)$?
\end{que}

We remark briefly that mapping class groups of surfaces with nonempty boundary do admit faithful actions on $\R$~\cite{HT1985,CB1988}.

\section{Braid groups and virtual specialness}

An important and often very difficult question in CAT(0) geometry is to determine whether or not a particular group or class of groups can be the fundamental group of a locally CAT(0) space, i.e. if they are \emph{CAT(0) groups}. Some of the first progress towards determining whether (general) Artin groups can be CAT(0) was made by Deligne~\cite{Deligne1972} and later by Charney--Davis~\cite{CharneyDavis1995}. Most mapping class groups cannot be CAT(0) groups, since this would violate the Flat Torus Theorem (cf. Subsections~\ref{subsec:solvable} and~\ref{subsec:nonex}). Braid groups do not have the obstructions to being CAT(0) groups that mapping class groups do, and in fact CAT(0) structures for braid groups have been found for up to 6 strands (see~\cite{HaettelKielak2016} for the most recent advance).

CAT(0) structures on braid groups as they have been produced are generally not virtual special cubulations, and in general we do not know the following:

\begin{que}\label{que:vspecial}
Are braid groups virtually special? Do braid groups embed as subgroups of right-angled Artin groups?
\end{que}

Question~\ref{que:vspecial} has attracted some attention recently and has been partially answered. Huang--Jankiewicz--Przytycki~\cite{HJP2016} give an if and only if condition for the existence of a cocompact cubulation of a two--dimensional Artin group. In the most recent version of his preprint, Haettel~\cite{Haettel2015} claims a complete classification of virtual cocompact cubulations of Artin groups, in particular showing that $B_n$ is not virtually cocompactly cubulated whenever $n\geq 4$.

A potential approach to resolving Question~\ref{que:vspecial} is through actions on one--manifolds:

\begin{que}\label{que:vbraid}
Let $G<B_n$ be a finite index subgroup, where $n\geq 4$. Does there exist an injection $G\to\Diff^{\infty}(\R)$?
\end{que}

\begin{prop}
Suppose the answer to Question~\ref{que:vbraid} is no. Then braid groups are not virtually special.
\end{prop}
\begin{proof}
Suppose that $B_n$ is virtually special. Then some finite index subgroup $G<B_n$ can be embedded as a subgroup of a right-angled Artin group $A(\gam)$. By Theorem~\ref{thm:raagr}, we have that $A(\gam)$ occurs as a subgroup of $\Diff^{\infty}(\R)$, a contradiction.
\end{proof}

\section{Higher dimensional actions}

In this final section, we consider the problem of finding right-angled Artin groups inside of diffeomorphism groups of manifolds in dimension at least two.

\begin{thm}\label{thm:surfaceraag}
Let $S$ be a two--manifold and let $\gam$ be a finite simplicial graph. Then $A(\gam)<\Diff^{\infty}(S)$.
\end{thm}
\begin{proof}
Let $\gam$ be given. By Theorem~\ref{thm:antitree}, there is a planar graph $T$ such that $A(\gam)<A(T^c)$, so it suffices to prove the result for graphs whose complements are planar.

So, fix a planar graph $T$ and an open disk $D\subset S$, and choose disks $D_v\subset D$ for each vertex of $T$ such that $D_v\cap D_w=\varnothing$ if and only if $\{v,w\}\notin E(T)$. Choose some number of marked points on $D$ and mapping classes $\{\psi_v\}_{v\in V(T)}$ as in the proof of Theorem~\ref{thm:crispwiest}, so that these mapping classes generate a copy of $A(T^c)$. We lift each $\psi_v$ to a $C^{\infty}$ diffeomorphism $\Psi_v$ supported on $D_v$. We claim that $\langle\Psi_v\mid v\in V(T)\rangle\cong A(T^c)$.

First note that there is a natural map \[f\colon A(T^c)\to\langle\Psi_v\mid v\in V(T)\rangle<\Diff^{\infty}(D)\] given by $v\mapsto \Psi_v$. Moreover, there is a natural surjection $g\colon \langle\Psi_v\mid v\in V(T)\rangle\to A(T^c)$ given by $\Psi_v\mapsto \psi_v$. Observe that the composition $g\circ f$ of these two maps is manifestly an isomorphism, whence $f$ is an isomorphism.
\end{proof}

It is not difficult to see that Theorem~\ref{thm:surfaceraag} can be improved to manifolds of arbitrary dimension $\geq 2$, by embedding a given right-angled Artin group in the group of $C^{\infty}$ diffeomorphisms of a small ball.

For a surface (or for higher dimensional manifolds), one may insist on group actions that preserve more structure than just a $C^{\infty}$ structure. For instance, if $(M,\omega)$ is a symplectic manifold, one may ask if a given group can be realized as a group of (Hamiltonian) symplectomorphisms of $M$. In this direction we have the following improvement of Theorem~\ref{thm:surfaceraag} due to M. Kapovich:

\begin{thm}[\cite{Kapovich2012}]\label{thm:raagsinham}
Let $\gam$ be a finite simplicial graph and let $(M,\omega)$ be a symplectic manifold. Then $A(\gam)$ embeds in the group of Hamiltonian symplectomorphisms of $M$.
\end{thm}

\section*{Acknowledgements}
The author is partially supported by Simons Foundation Collaboration Grant number 429836. The author is grateful to the referee for pointing out several references and for noting several corrections, and for generally improving the manuscript.





\bibliographystyle{amsplain}
\bibliography{ref}

\end{document}